\documentclass[a4paper]{amsart}

\usepackage{amsmath, amssymb, amsthm, color, multicol, mathabx, verbatim}

\newcommand{\SL}[2][\mathbb{Z}]{\mathrm{SL}_{#2}(#1)}
\newcommand{\R}{\mathbb{R}}
\newcommand{\Z}{\mathbb{Z}}
\newcommand{\N}{\mathbb{N}}

\newcommand{\Supp}[1]{\mathrm{Supp}(#1)}

\newcommand{\PalWidth}{\mathrm{PW}}
\newcommand{\PW}{\mathrm{PW}}
\newcommand{\x}{\mathbf{x}}

\newcommand{\uvec}{\mathbf{u}}
\newcommand{\evec}{\mathbf{e}}
\newcommand{\vvec}{\mathbf{v}}

\newtheorem{thm}{Theorem}[section]
\newtheorem{cor}[thm]{Corollary}
\newtheorem{lemma}[thm]{Lemma}

\newtheorem{question}[thm]{Question}

\newtheorem{rem}[thm]{Remark}

\title[Palindromic width of wreath products and solvable groups]{Palindromic width of wreath products, metabelian groups,  and  max-n solvable groups}
\author{T.{} R.{} Riley and A.{} W.{} Sale}

\thanks{The first author gratefully acknowledges partial  support from NSF grant DMS--1101651.}

\begin{document}

\begin{abstract}
A group has finite palindromic width if there exists  $n$ such that every element can be expressed as a product of $n$ or fewer palindromic words. We show that if $G$ has finite palindromic width with respect to some generating set, then so does $G \wr \Z^{r}$.   We also  give a new, self-contained, proof that finitely generated metabelian groups have finite palindromic width. Finally, we  show that solvable groups satisfying the maximal condition on normal subgroups (max-n) have finite palindromic width.

 \medskip

\footnotesize{\noindent \textbf{2010 Mathematics Subject
Classification:  20F16, 20F65}  }
\end{abstract}

\maketitle

\section{Introduction}

The \emph{width} of a group $G$ with respect to an (often infinite) generating set $A$ is the minimal $n$ such that every $g\in G$ can expressed as the product of $n$ or fewer elements from $A$. If no such $n$ exists,   the width is infinite.
Examples include the primitive width of free groups (e.g.\  \cite{BST05}), and the   commutator width of a derived subgroup, or more generally the verbal width of a verbal subgroup with respect to any given word (\cite{Segal09} is a survey).

This paper concerns \emph{palindromic width}.    Suppose $G$ is a group with  generating set $X$.
Write $\PalWidth(G, X)$ for the width of $G$ with respect to the set of palindromic words on ${X\cup X^{-1}}$ --- the words that read the same forwards as  backwards.

We give  bounds on palindromic width in a variety of settings. Here is the first.  (We view $G$ and the $\Z^r$-factor as subgroups of $G \wr \Z^r$ in the standard way.)

\begin{thm} \label{wr thm}
If $G$ is a group with finite generating set $A$, then 
$$\PalWidth(G \wr \Z^r, A \cup S) \ \leq \  3r + \PalWidth(G,A)$$
where S is the standard generating set of $\Z^r$. Better, when $r=1$,
$$\PalWidth(G \wr \Z , A \cup \{t\}) \ \leq \ 2 + \PalWidth(G,A)$$
where $t$ is a generator of $\Z$.
\end{thm}

The upper bound of our next theorem is a corollary.

\begin{thm}\label{cor:ZwrZ}
The palindromic width of  $$\Z \wr \Z \ = \  \left\langle  \, a, t \, \left| \,   \left[a,a^{t^k}\right]=1 \ (k \in \Z)  \,  \right. \right\rangle$$  with respect to $a,t$ is  $3$.
\end{thm}

The heart of our proof of Theorem~\ref{wr thm} is a result (Lemma~\ref{lem:r-dim symmetric functions})  on expressing  finitely supported functions from $\Z^r $  to a group as a pointwise product of two such functions both exhibiting certain symmetry.  We develop this result  (in Section~\ref{skew-symmetric functions section})  to more elaborate results on expressing finitely supported functions from $\Z^r$ to a ring as the sum of what we call \emph{skew-symmetric} finitely supported functions.   This led us to a new proof of the following theorem which we have since discovered was proved by Bardakov \& Gongopadhyay not long prior.

\begin{thm}[Bardakov \& Gongopadhyay~\cite{BG13}]\label{thm:main}
The palindromic width of any metabelian group with respect to any finite generating set is finite.  
\end{thm}

Our proof is self-contained. 
 Bardakov \& Gongopadhyay use a result of Akhavan-Malayeri and Rhemtulla \cite{AMR98}, which in turn uses a result from the unpublished PhD thesis of Stroud \cite{Stro66}, details of which may also be found in \cite{Segal09}. However they established more, namely that free abelian-by-nilpotent groups have finite palindromic width.  In their sequel \cite{BGnilpotentpals2}, Bardakov \& Gongopadhyay have investigated lower bounds for the palindromic width of nilpotent groups and abelian--by--nilpotent groups.   
  Also using work of Akhavan-Malayeri concerning the nature of commutators \cite{Akha2010wreath}, E.{} Fink claims that the wreath product of a finitely generated free group with a finitely generated free abelian group,  and hence also the wreath product of any finitely generated group with a finitely generated free abelian group, has finite palindromic width, \cite{FinkPalindromes}. 

\emph{Boundedly generated}  groups  provide many examples with finite palindromic width.  A group  $G$ is boundedly generated when there exist ${a_1, \ldots, a_k \in G}$ such that every element can be expressed as $a_1^{r_1} \cdots a_k^{r_k}$ for some ${r_1, \ldots, r_k \in \Z}$.    In such groups, $\PalWidth(G, \{a_1, \ldots, a_k\}) \leq k$.  They include all finitely generated solvable minimax groups  \cite{Kr84} (and so  all finitely generated nilpotent or, more generally, polycyclic  groups),  a  non-finitely presentable example of Sury \cite{Su97}, and $\SL{n}$ for $n \geq 3$  with respect to elementary matrices \cite{CK83} and generalizations \cite{Mu95,Ta90}.   All finitely presented, torsion-free, abelian-by-cyclic groups (and so all solvable Baumslag--Solitar groups)  are boundedly generated:  by \cite{BS78} (see also \cite[\S1.1]{FM01})  they have presentations  
$$\langle \, t, a_1, \ldots , a_m \, \mid  \, a_i a_j=a_j a_i, \ ta_i t^{-1} = w_i(a_1,\ldots ,a_m), \ \forall i,j  \, \rangle$$ and each element can be represented as $t^{-i}a_1^{r_1}\cdots a_m^{r_m}t^j$ for some  non-negative integers  $i,j$ and some $r_1, \ldots, r_m \in \Z$.  And $$\Z_2 \ast \Z_2  \ = \  \langle x,y \mid x^2=y^2=1 \rangle$$  is boundedly generated as every element is expressible as   $(xy)^l$, $(xy)^lx$, or $y(xy)^l$ for some $l \in \Z$.  Passing to or from subgroups of finite index preserves bounded generation \cite[Exercise~4.4.3]{BHV}, as does passing to a quotient.

 There are finitely generated metabelian groups which are not boundedly generated, for example,  $\Z \wr \Z$ \cite{ENS}.   So: 

\begin{cor}\label{cor:main}
There are finitely generated groups with finite palindromic width (with respect to all finite generating sets) which are not boundedly generated.   
\end{cor}

 A group $G$ satisfies the maximal condition for normal subgroups (max-n)  if for every normal subgroup $N$ of $G$, there is a finite subset which normally generates $N$. Finitely generated (abelian-by-polycyclic)-by-finite groups are examples \cite{Hall54}.
 We extend  Theorem~\ref{thm:main} to:

\begin{thm}\label{thm:max n}
If $G$ is a finitely generated solvable group satisfying max-n, then $G$ has a finite generating set $B$ such that $\PalWidth(G,B)$ is finite.
\end{thm}

We understand that this result has been proved independently by Bardakov and Gongopadhyay \cite{BGsolvablepals}, where they in fact show that a finitely generated solvable group which is abelian-by-(max-n) has finite palindromic width. This therefore includes the finitely generated solvable groups of derived length 3. In the same paper they also provide a different proof of Theorem \ref{cor:ZwrZ}, first showing that $\Z \wr \Z$ has commutator width 1.

There are groups   known to have infinite palindromic width: rank-$r$ free groups $F(x_1, \ldots, x_r)$  for  all $r \geq 2$  \cite{BST05} with respect to $\{x_1, \ldots, x_r\}$ and, with the sole exception of  $\Z_2 \ast \Z_2$, all free products $\Asterisk_{i=1}^m G_i$ of non-trivial   groups  with respect to $\bigcup_i G_i$,  \cite{BT06}.   The proofs  in  \cite{BST05} and  \cite{BT06} are novel in that they use quasi-morphisms.

The structure of the paper is as follows. In Section \ref{ZwrZ} we consider the palindromic width of groups $G \wr \Z^r$, proving Theorem \ref{wr thm}. Section \ref{lower bound} gives the precise value for the palindromic width of $\Z \wr \Z$ (Theorem \ref{cor:ZwrZ}). Our proof for the result concerning metabelian groups is contained in Section \ref{metabelian groups section}, while Section \ref{sec:max-n} deals with solvable groups satisfying max-n.  We conclude with a discussion of   open questions in Section~\ref{Open questions}.

\medskip

\emph{Acknowledgements.}  We thank Valeriy Bardakov, Krishnendu Gongopadhyay,  Elisabeth Fink, and an anonymous referee for their comments.

\section{The palindromic width of $G \wr \Z^r$} \label{ZwrZ}

Suppose $G$ is a group with finite generating set $A$.  The group $G \wr \Z^r=  \left( \bigoplus_{\Z^r}G  \right) \rtimes \Z^r$, where we view  $\bigoplus_{\Z^r}G$  as the group of finitely supported functions $\Z^r \to G$ under coordinatewise multiplication, and  elements  $\mathbf{v}$ of the $\Z^r$-factor act  on  $\bigoplus_{\Z^r}G$ by the \emph{shift operation}:  $f^{\mathbf{v}}(\x) = f(\x-\mathbf{v})$ for all $\x \in \Z^r$.  Note that $G\wr \Z^r$ is generated by the union of  $A \times \{ \mathbf{0}\}$  and $\{1\} \times B$ where $B$ is the standard basis $B = \{\evec_1,\ldots , \evec_r\}$ for $\Z^r$.

\subsection{An example from $\Z \wr \Z$} \label{subsec:example}

Define $\Delta_{1} : \Z \to \Z$ to be  $1$ at $0$ and  $0$ elsewhere, and  $\mathbf{0} : \Z \to \Z$ to be everywhere $0$.   The standard  generating set for  $\Z \wr \Z$ is $\{ a,t \}$ where   $a = (\Delta_{1}, 0)$ and $t = (\mathbf{0}, 1)$.

Let $f : \Z \to \Z$ be the function whose non-zero values are given in the following table.   
$$\begin{array}{l|p{0.35cm}p{0.35cm}p{0.35cm}p{0.35cm}p{0.35cm}p{0.35cm}p{0.35cm}p{0.35cm}p{0.35cm}p{0.35cm}p{0.35cm}p{0.35cm}p{0.35cm}p{0.35cm}p{0.35cm}}
x 	& -7	&-6	&-5	&-4	&-3	&-2	&-1	& 0	& 1	& 2	& 3	& 4	& 5 & 6 & 7 \\ \hline
f(x)& 0		& 0 & 0 & 3 &-1 & 4 & 0 & 0 & 1 & 5 & 0 & 0 & 0 & 0 & 2 \\  
g(x)& 0		& -2&-2 & 1 & 0 & 4 &-1 &-2 &-1 & 4 & 0 & 1 &-2 &-2 & 0 \\  
h(x)& 0 	& 2 & 2 & 2 &-1 & 0 & 1 & 2 & 2 & 1 & 0 &-1 & 2 & 2 & 2 
\end{array}.$$ 
We will explain how to express $(f,7)$ as a product of three palindromes.  The ideas apparent here are the core of the proof of our upper bounds on the palindromic width of general  $G \wr \Z^r$.

The table above also shows the non-zero values of functions $g,h: \Z \to \Z$ such that $f=g+h$ with $g$ being symmetric about zero (that is, $g(x) = g(-x)$ for all $x$) and $h$ being symmetric about $\frac{1}{2}$ (that is, $h(x) = h(1-x)$ for all $x$).

Here is how these $g$ and $h$ were found.  We sought $g$ supported on 
 $\{-7,-6,\ldots, 6,7\}$ and $h$ supported on 
 $\{-6,-5,\ldots, 6,7\}$, which was reasonable since the  support of $f$ is a subset of each  these sets and they are symmetric about $0$ and $\frac{1}{2}$, respectively.

Since $-7$ is outside the support of $h$ we have $h(-7)=0$. We deduced from the requirement that $f=g+h$ that $g(-7)=0$. 
From here,  symmetry of $g$ about $0$ and of $h$ about $\frac{1}{2}$, together with  $f=g+h$,  determined the remaining values taken by $g$ and $h$. The symmetry allowed us  to ``jump'' to the other side of $0$ or $\frac{1}{2}$. Once there we applied $f=g+h$ and were ready to ``jump'' again.

Specifically in this example symmetry of $g$ about $0$ gave $g(7)=g(-7)=0$, and then  $h(7)=f(7)-g(7)=2$. Then, using the symmetry of $h$ about $\frac{1}{2}$, we found $h(-6)=h(6)=2$, and then $g(-6)=-2$ (as   $f=g+h$). We   continued in this way until the table was complete.

As we will shortly explain, the palindromes
\begin{eqnarray*}
w_g & := & t^{-6} a^{-2}ta^{-2}tat^2a^4ta^{-1}ta^{-2}ta^{-1}ta^4t^2ata^{-2}ta^{-2}t^{-6} \\ 
w_h & := &  t^{-6}a^2ta^2ta^2ta^{-1}t^2ata^2ta^2tat^2a^{-1}ta^2ta^2ta^2t^{-6}.
\end{eqnarray*}
represent $(g,0)$ and $(h,1)$, respectively.  
So the product $w_gw_ht^6$  of three palindromes represents $(f,7)$.

How we obtained $w_g$ and $w_h$ is best explained using the \emph{lamplighter model} of $\Z\wr\Z$. 
View the real line as an infinite street and imagine a lamp at each integer point. Each lamp has $\Z$--many states. A \emph{lamp configuration} is an assignment of a state (an integer) to each lamp with all but finitely many lamps assigned zero.  (Equivalently a lamp configuration is a  finitely supported function $\Z\to\Z$.) Elements of $\Z\wr\Z$ are   represented by a lamp configuration together with a choice of lamp by which a lamplighter is imagined to stand. The identity is represented by  the street in darkness, all lamps are in state zero, with the lamplighter at position zero.  

The generators $t$ and $a$ of $\Z\wr\Z$ act  in the following manner: applying $t$  moves the lamplighter one step right  (that is, in the positive direction); applying $a$ adds one to the state of the bulb at the location of the lamplighter.

Group elements that can be represented by palindromes on $\{a^{\pm 1}, t^{\pm 1}\}$  can be recognised as those  for which  the lamp configuration is symmetric about some point $m/2$, where $m \in \Z$, and such that the lamplighter finishes at position $m$.

For example, in the instance of $(g,0)$  we begin by sending the lamplighter from zero to the leftmost extreme of the support of $g$. This is 6 steps left, so is an application of $t^{-6}$. Next we change the state of the bulb here to $g(-6)=-2$, so we apply $a^{-2}$. Now we proceed right one step at a time by applying $t$. After each step we adjust the state of the bulb according to the function $g$ by applying $a$ or $a^{-1}$ the appropriate number of times. When we reach the rightmost extreme of the support of $g$, we will have a lamp configuration which is symmetric about $0$. To finish, and give ourselves a palindrome, we need to repeat the first steps we took in reaching the left-most point of the support from the origin, namely we apply $t^{-6}$ again.

\subsection{The upper bounds}

Suppose $G$ is a group with generating set $A$.  In the  example above we expressed $f : \Z \to \Z$ as a sum of two symmetric functions. We will generalise  this to expressing functions $f:\Z^r \to G$ as products of symmetric functions.   

Let $(A \cup A^{-1})^\ast$ denote the free monoid on the set $A \cup A^{-1}$---that is, the (finite) words on $A$ and $A^{-1}$.  For  $u$ in this free monoid, let $\overline{u}$ denote the same word written in reverse.
Let $\varepsilon$ denote the empty word.
For $\gamma \in G$, define   $\Delta_\gamma:\Z^r \to (A \cup A^{-1})^{\ast}$ by $\Delta_\gamma(\mathbf{0}) = \gamma$ and $\Delta_\gamma(\x) = \varepsilon$ for all $\x \neq \mathbf{0}$.  
Let $\evec_1,\ldots,\evec_r$ be the standard generating set for $\Z^r$.
 
\begin{lemma}\label{lem:r-dim symmetric functions}
Suppose $f : \Z^r \to (A \cup A^{-1})^\ast$ is finitely supported. Then $$f \ = \ \Delta_\gamma f_0f_1\ldots f_r$$ in $\bigoplus_{\Z^r}G$ for some $\gamma \in G$ and some finitely supported $$f_0,f_1, \ldots, f_r : \Z^r \to (A \cup A^{-1})^\ast$$ such that $f_0(\x)=\overline{f_0(-\x)}$ and $f_i(\x)=\overline{f_i(\evec_i-\x)}$ for $i = 1, \ldots , r$.
\end{lemma} 

\begin{proof}
We induct on $r$. For $r=1$, take $n >0$ such that $f(j) =\varepsilon$ for all $j$ for which $|j| >n$.  Define $g$ and $h$ as follows. First set $h(-n) :=  \varepsilon$.  Then, for $i=n, \ldots , 1$, define
$$\begin{array}{lllll}
\overline{g(i)}	& := 	& g(-i) 	& := & f(-i)  h(-i)^{-1} \\
h(i) 	& := 	& \overline{h(1-i)} & := &  g(i)^{-1} f(i) 
\end{array}$$
and finish by setting $g(0):= \varepsilon$ and $g(j) = h(j) =\varepsilon$ for all $j$ for which $|j| >n$.  Then take $\gamma :=  f(0) h(0)^{-1} $. By construction, $f_0:=g$ and $f_1 :=h$ have the required properties.

Now suppose $r>1$.
Let $n>0$ be such that $f$ is supported on $\{-n, \ldots , n-1, n\}^r$. We first define functions $g,h : \Z^r \to (A \cup A^{-1})^\ast$ such that $f=gh$, with $h$ carrying the required symmetry about $\frac{1}{2}\evec_r$ and the symmetry of $g$ about the origin being satisfied everywhere except in the codimension 1 subspace orthogonal to $\evec_r$.

Consider $\Z^r$ as $\Z^{r-1}\times \Z$, where the $1$--dimensional factor is the span of $\evec_r$. 
For $\x \in \Z^{r-1}$, recursively define $g(\x,i)$ and $h(\x,i)$ for $i=1,\ldots, n$ and $g(-\x,i)$ and $h(-\x,i)$ for $i=0,-1,\ldots, -n$ as follows. First set $h(-\x,-n):=\varepsilon$. Then, for $i=n,\ldots, 1$, define
$$\begin{array}{lllll}
\overline{g(\x,i)}	& := 	& g(-\x,-i) 	& := & f(-\x,-i)h(-\x,-i)^{-1} \\
h(\x,i) 	& := 	& \overline{h(-\x,1-i)} & := & g(\x,i)^{-1}f(\x,i)
\end{array}$$
and
$$g(\x,0):=f(\x,0)h(\x,0)^{-1}.$$

Since $g$ and $h$ are supported on $\{-n,\ldots, n-1, n\}^r$, this defines them everywhere on $\Z^r$. Let $S:=\{(\x,0) \mid \x \in \Z^{r-1}\}$. Note that $h(\x,i)=\overline{h(-\x,1-i)}$ for all $(\x,i) \in \Z^r$ and so $f_r:=h$ satisfies the condition required by the lemma. However, we only know that  $g(\x,i)=\overline{g(-\x,-i)}$ for all $(\x,i) \in \Z^r \smallsetminus S$, and so $g$ cannot serve as $f_0$ as it stands.

 By the inductive hypothesis, the restriction of $g$ to $S$ can be expressed as the product of $\Delta_\gamma$, for some $\gamma \in G$, and suitably symmetric functions $f_0,f_1, \ldots, f_{r-1} : \Z^{r-1} \to (A \cup A^{-1})^\ast$. Extend these to be functions on $\Z^r$ by defining them to be $\varepsilon$ outside of $S$. Note that they retain their symmetry. The product $f=\Delta_\gamma f_0f_1\ldots f_{r-1} f_r$  is therefore the required expression.
\end{proof}

 We are now ready to prove the upper bounds of Theorem~\ref{wr thm}.  We will  show that if  $x_1, \ldots, x_r$ are generators for $\Z^r$, which will we now write multiplicatively, then $\PalWidth(G\wr \Z^r,A\cup\{x_1,\ldots, x_r\}) \leq 3r+\PalWidth(G,A).$

  We will explain how to write an arbitrary element $(f,x_1^{t_1}\ldots x_r^{t_r}) \in G \wr \Z^r$ as a product of at  most $3r+\PalWidth(G,A)$ palindromes. Choose $n>0$ so that $f$ is supported on $\{-n,\ldots n-1,n\}^r$. Iteratively define  palindromes $u_i$, for $i=1,\ldots, r$, as follows:
 $$\begin{array}{rcl}
 u_1 &:=& x_1^{2n},\\
 u_2 &:=& (u_1x_2u_1^{-1}x_2)^n u_1,\\
 &\vdots&\\
 u_{r} &:=& (u_{r-1} x_{r} u_{r-1}^{-1} x_{r})^n u_{r-1}.
 \end{array}$$
So $u_r$ defines a Hamiltonian path which snakes around the cube $[0,2n]^r$. Set $u:=x_1^{-n}\ldots x_r^{-n} u_r x_r^{-n} \ldots x_1^{-n}$.  Insert  the words $f(\x)$ for all  $\x \in \{-n,\ldots , n-1,n\}^r$  into $u$ as follows. For each such $\x$ there exists a unique prefix $u(\x)$ of $u_r$ such that $x_1^{-n}\ldots x_r^{-n} u(\x) = \x$. Rewrite $u$ by inserting $f(\x)$ after $x_1^{-n}\ldots x_r^{-n} u(\x)$, for each such $\x$. Denote the resulting word by $u^f$.
 
 Suppose $g:\Z^r \to (A \cup A^{-1})^\ast$ satisfies $g(\x)=\overline{g(-\x)}$ for all $\x \in \Z^r$. Then $u^g$ will be a palindrome and will represent $(g,0)$ in $G \wr \Z^r$.
 
 We wish to  produce words in a similar manner for a function $h:\Z^r \to (A \cup A^{-1})^\ast$ which satisfies $h(\x) = \overline{h(\evec_i-\x)}$, for all $\x \in \Z^r$. After permuting the basis vectors, we may assume $i=1$. Define palindromes $v_i$ as follows:
 $$\begin{array}{rcl}
 v_1 &:=& x_1^{2n+1},\\
 v_2 &:=& (v_1 x_2 v_1^{-1} x_2)^n v_1,\\
 &\vdots&\\
 v_{r} &:=& (v_{r-1} x_{r} v_{r-1}^{-1} x_{r})^n v_{r-1},
 \end{array}$$
where $n$, as before, is taken so that $\textup{Supp}(h) \subseteq \{-n,\ldots , n-1, n\}^r$. 
 Set $v:=x_1^{-n}\ldots x_r^{-n} v_r x_r^{-n} \ldots x_1^{-n} x_1^{-1}$, which is not a palindrome, but rather is the product of two palindromes,   the second   being $x_1^{-1}$.
 
 We rewrite $v$, as we did for $u$, by inserting $h(\x)$ at the appropriate points to give a word $v^h$.  The symmetric properties held by $h$ mean that when we insert $h(\x)$ and $h(\evec_1-\x)$ into the appropriate places of $v$ we will still have a product of two palindromes, with the second palindrome being $x_1^{-1}$ as was originally the case. Thus  the word $v^h$ will be the product of two palindromes and moreover will represent $(h,0)$ in $G \wr \Z^r$.  
 
 Express the function $f$ as per Lemma \ref{lem:r-dim symmetric functions} as $f=\Delta_\gamma f_0f_1\ldots f_r$. In $G\wr \Z^r$ the element $(f_0,0)$ is represented by the palindrome $u^{f_0}$, and each $(f_i,0)$, for $i=1, \ldots , r$, is represented by the product $v^{f_i}$ of two palindromes. Let $\pi$ be a product of at most $\PW(G,A)$ palindromes representing $\gamma$.   Then
 		$$(f,0) \ = \  \pi u^{f_0} v^{f_1} \ldots v^{f_r}$$
   is the product of at most $\PW(G,A)+2r+1$ palindromes. Finally, we post-multiply by $x_r^{t_r}\ldots x_1^{t_1}$ to obtain a word for $(f,x_1^{t_1}\ldots x_r^{t_r})$. Since the second palindrome of $v^{f_r}$ is $x_r^{-1}$, this is absorbed into the first palindrome of $x_r^{t_r}\ldots x_1^{t_1}$. Thus we obtain a word representing $(f,x_1^{t_1}\ldots x_r^{t_r})$ which is the product of $\PW(G,A)+3r$ or fewer palindromes.

When $r=1$ we can modify our proof of Lemma~\ref{lem:r-dim symmetric functions}  to obtain a stronger upper bound for $\PalWidth(G\wr \Z, A \cup\{t\})$ as follows.
 
We construct $g$ and $h$ from $f$ as in Lemma \ref{lem:r-dim symmetric functions}, but with one difference. We absorb one palindrome from an expression for $\gamma$ into $g(0)$, which had been taken to be the empty word in the proof of the lemma.

Suppose that $f(0) h(0)^{-1} = w_1 \ldots w_k$, where $k$ is minimal such that each $w_i$ is a palindrome. Set $g(0)=w_k$, which is allowed since $w_k$ is a palindrome. Then $\gamma:=f(0) h(0)^{-1}g(0)^{-1}$ can be expressed as the product of $k-1$ palindromes. In particular,  $k\leq \PalWidth(G,A)$, leading to:

$$\PalWidth(G \wr \Z, A  \cup \{t\}) \ \leq  \ 2+\PalWidth(G, A ).$$

\section{The palindromic width of $\Z \wr \Z$ is at least $3$} \label{lower bound}

Here we show that  $\PalWidth(\Z\wr\Z,\{a,t\}) \geq 3$ and so complete our proof of Theorem~\ref{cor:main}.

Let $f:\Z \to \Z$ have support $\{0,1\}$ and suppose $f(0) \neq f(1)$. We will show that $(f,3)$ cannot be  expressed as the product of two palindromic words.

First note that $(g,r) \in \Z \wr \Z$ can be expressed as a palindrome if and only if $g(x)=g(r-x)$ for all $x \in \Z$. Indeed, given a palindrome on $\{a,t\}$, the lamp configuration obtained from this word must be symmetric about $\frac{1}{2}r$  (see Section \ref{subsec:example}), implying that the corresponding function must be symmetric about this point, as required. Conversely, if $g$ is symmetric about $\frac{1}{2}r$, then we can construct a palindrome in which the lamplighter will run first to the smallest lamp in the support, and then run in the positive direction, turning on all lamps to the appropriate configuration, and finishing off by running to $r$. The reader may check that the word obtained from this journey is indeed a palindrome.

Suppose there exists $p,q \in \Z$ and $g,h:\Z \to \Z$ such that $g$ is symmetric about $\frac{1}{2}p$, $h$ is symmetric about $\frac{1}{2}q$ and $(f,3)=(g,p)(h,q)$. Let $h_0$ be the shift of $h$ by $p$---that is, $h_0(x)=h(x-p)$ for $x \in \Z$. Thus $p+q=3$ and 
\begin{equation}\label{eq:bounce}
g(x)+h_0(x)=f(x)
\end{equation}
for all $x \in \Z$. We will show that at least one of $g$ or $h_0$ (hence $h$) must have infinite support.

We claim that $g(x) = g(x+3)$ except possibly when $x \in \{-3,-2,p-1,p\}$. We use the equalities:

\begin{equation}\label{eq:g right}
g(x)  \ = \   g(p-x) \  = \   -  h_0(p-x)  \ = \    -  h_0(x+3)  \ = \    g(x+3)  
\end{equation}
which follow from, in order, firstly symmetry of $g$ through $\frac{1}{2}p$, secondly   equation \eqref{eq:bounce}  assuming $f(p-x)=0$, thirdly   symmetry of $h_0$ through $p+\frac{1}{2}q$, and finally   a second application of equation \eqref{eq:bounce} assuming $f(x+3) =0$. As $\Supp{f} = \{0,1\}$, this  can fail only if $x \in \{-3,-2,p-1,p\}$.

Similarly,  
\begin{equation}\label{eq:h left}
h_0(x) \ = \   h_0(3+p-x) \ = \   - g(3+p-x)  \ = \   - g(x-3) \ = \  h_0(x-3)  
\end{equation}
provided $3+p-x$ and $x-3$ are not in $\{0,1\}$---that is,
provided $x \notin \{3,4,p+2,p+3\}$.

First suppose $p<0$. Then $g(x)=0$ for all $x \geq 0$, otherwise it will have infinite support. Symmetry of $g(x)$ about $\frac{1}{2}p$ then implies that $g(x)=0$ for $x \leq p$. Applying equation \eqref{eq:bounce} gives $\Supp{h_0} \subseteq \{p+1,\ldots , 1\}$ and $h_0(x)=f(x)$ for $x = 1,2$. 
Equation \eqref{eq:h left} implies that
$$h_0(x) \ = \  \left\{\begin{array}{ll}
f(0) & \textrm{if $x \equiv 0 \ (\mathrm{mod} \ 3)$,}\\
f(1) & \textrm{if $x \equiv 1 \ (\mathrm{mod} \ 3)$,}\\ 
0 & \textrm{if $x \equiv 2 \ (\mathrm{mod} \ 3)$}
\end{array}\right.$$  for $x \in \{p+1,\ldots , 1\}$: the values of $h_0$ at $0$, $1$ and $2$, determine those at  $-3$, $-2$ and $-1$, respectively, and then $-6$, $-5$ and $-4$, and so on until the value at $p$ which we cannot deduce from that at $p+3$.   
But then $h_0$ cannot be symmetric about any point: if  a  function $\Z \to \Z$ repeats a length-$3$ (or greater) pattern of distinct values on  an interval of length at least 3, and is zero elsewhere, then $f$ cannot be symmetric. (In the case $p=-1$, we use the fact that $\Supp{h_0}=\{0,1\}$ and $f(0) \neq f(1)$.)

 Now suppose $p\geq 0$. Then by equation~\eqref{eq:h left}, for $x \leq 1$, if $h_0(x) \neq 0$ then $h_0(x-3) \neq 0$. Thus,   finite support of $h_0$ implies  that $h_0(x) = 0$ for $x \leq 1$. Symmetry about $p+\frac{1}{2}q =\frac{1}{2}p + \frac{3}{2}$ then gives $h_0(x)=0$ for $x \geq p+2$.  In particular, by equation \eqref{eq:bounce}, $g(x)=f(x)$ for $x \leq 1$ and $x \geq p+2$, hence $\Supp{g} \subseteq\{0,\ldots , p+1\}$. Also, by \eqref{eq:g right}, for $x \in \{0,\ldots , p+1\}$, we have
$$g(x)= \left\{\begin{array}{ll}
f(0) & \textrm{if $x \equiv 0 \ (\mathrm{mod} \ 3)$;}\\
f(1) & \textrm{if $x \equiv 1 \ (\mathrm{mod} \ 3)$;}\\ 
0 & \textrm{if $x \equiv 2 \ (\mathrm{mod} \ 3)$.}
\end{array}\right.$$
As before, such a function cannot be symmetric. (When $p=0$ we use that $f(1) \neq 0$.)  

This covers all possible values of $p$, so $(f,3)$ cannot be expressed as the product of two palindromes.

\section{Finite palindromic width of metabelian groups} \label{metabelian groups section}

In this section we give our proof of Theorem~\ref{thm:main}. 

Let $F = F(x_1, \ldots, x_r)$ be a free group on $r$ generators and $F''$ be its second derived subgroup.  Then $F / F''$ is the \emph{free metabelian group of rank $r$}.   

The property of having finite palindromic width   passes to quotients.  Indeed, if  $G$ is a group with generating set $X$ and $\overline{G}$ is a quotient, then  $${\PalWidth(\overline{G},\overline{X}) \  \leq \  \PalWidth(G, X)},$$ where $\overline{X}$ is the image of $X$ under the quotient map.  So, it will suffice to prove Theorem~\ref{thm:main} for finitely generated \emph{free} metabelian groups with respect to their standard generating sets.  More precisely, we will prove 

that the palindromic width of the free metabelian group $F/F''$ of rank $r$ with respect to $x_1, \ldots, x_r$ is at most $2^{r-1}r(r+1)(2r+3)+4r+1.$
Bardakov \& Gongopadhyay give a better bound of $5r$ in \cite{BG13}. 

Our proof 
begins with a pair of lemmas  in Section~\ref{skew-symmetric functions section} on expressing finitely supported functions on $\Z^r$ as the sum of what we call \emph{skew-symmetric} finitely supported functions.    Then in  Section~\ref{sec:palindromic width} we determine a relationship between skew-symmetric functions and palindromes in the subgroup $F'/F''$ of the free metabelian group, which leads to the theorem. 

\subsection{Skew-symmetric functions on $\Z^r$}  \label{skew-symmetric functions section}

When dealing with $G \wr \Z$ in Section~\ref{ZwrZ}, we saw how palindromes were closely related to the symmetry of functions $f:\Z \to \Z$. However, when instead investigating the free metabelian groups, we shall relate palindromes to what we call  \emph{skew-symmetric} functions on $\Z^r$.  When $r=1$ these are the functions that are translates of odd functions.  In general, we say that a function $f$ from $\Z^r$  to a ring $R$ is \emph{skew-symmetric} if  there exists $\mathbf{p} \in \frac{1}{2}   \Z^r$ such that for all $\mathbf{x} \in \Z^r$,  $f(\mathbf{x}) = - f(2\mathbf{p} -\x)$---that is, its values at $\x$  and at the reflection of $\x$ in  $\mathbf{p}$ sum to zero. Note that, when $\mathbf{p} \in   \Z^r$, this condition at    $\mathbf{x} = \mathbf{p}$ is that    $2f(\mathbf{p})=0$.

Let $\mathbf{e}_1, \ldots, \mathbf{e}_r$ denote the standard basis of unit-vectors for $\R^r$.  The following lemma is for a ring $R$ and is written additively as appropriate for its forthcoming application, but we remark that the proof given works with $R$ replaced by any  abelian group.

\begin{lemma}\label{lem:skew symmetric functions half}
For all  $r \geq 1$ and all $\mathbf{p} \in \frac{1}{2}    \Z ^r $, every finitely supported function  $f:\Z^r \to R$ such that $\sum_{\mathbf{x} \in \Z^r} f(\mathbf{x})=0$  is the sum of $r+1$  finitely supported functions skew-symmetric  about $\mathbf{p}, \mathbf{p} + \frac{1}{2} \mathbf{e}_1, \ldots, \mathbf{p} + \frac{1}{2} \mathbf{e}_r$. 
\end{lemma}

\begin{proof}
We will prove the  result when every entry in $\mathbf{p}$ is  $0$ or $\frac{1}{2}$.  This suffices as, if the result holds for a given $f$, then it holds for all its translates.  

Take an integer $n>0$ such that  $f$ is supported on $\{-n,\ldots, n-1, n\}^r$.   We will define functions $g,h : \Z^r \to R$, both supported on  $\{-n,\ldots, n-1, n\}^r$, such that $f=g+h$.  View $f$, $g$ and $h$ as functions $\Z^{r-1} \times \Z \to R$.  Let $\overline{\mathbf{p}}$ be the projection of $\mathbf{p}$   to the $\Z^{r-1}$--factor.   So $\mathbf{p}$ is $(\overline{\mathbf{p}},0)$ or  $(\overline{\mathbf{p}},\frac{1}{2})$.

First suppose $\mathbf{p}=(\overline{\mathbf{p}},0)$. For $\x \in \Z^{r-1}$, recursively define $g(\x,i)$ and $h(\x,i)$ for $i =1, \ldots, n$ and  $g(2\overline{\mathbf{p}} -\x,i)$ and $h(2\overline{\mathbf{p}}-\x,i)$ for $i =0, -1, \ldots, -n$ as follows.  First set $h(2\overline{\mathbf{p}}-\x,-n) 	 :=  0$.  Then, for $i=n, \ldots 1$, define
$$\begin{array}{lllll}
-g(\x,i)	& := 	& g(2\overline{\mathbf{p}}-\x,-i) 	& := & f(2\overline{\mathbf{p}}-\x,-i) - h(2\overline{\mathbf{p}}-\x,-i) \\
-h(\x,i) 	& := 	& h(2\overline{\mathbf{p}}-\x,-i+1) & := & -f(\x,i) + g(\x,i)
\end{array}$$
and
$$g(2\overline{\mathbf{p}}-\x,0) 	 :=  f(2\overline{\mathbf{p}}-\x,0) - h(2\overline{\mathbf{p}}-\x,0).$$

(In the case $r=1$, the $\x$ and $2\overline{\mathbf{p}}-\x$ terms are absent.)  As $g$ and $h$ are supported on  $\{-n,\ldots, n-1, n\}^r$, this defines them everywhere on $\Z^r$.  Let ${S  := \{   (\x,0)   \mid   \x \in \Z^{r-1}   \}}$. Observe that  $h$ is skew-symmetric about $(\overline{\mathbf{p}}, \frac{1}{2})=\mathbf{p}+\frac{1}{2}\mathbf{e}_r$ and that  $g$  is \emph{nearly} skew-symmetric about $(\overline{\mathbf{p}},0)=\mathbf{p}$, the condition only (possibly) failing on $S$. 

Suppose, on the other hand, $\mathbf{p}=(\overline{\mathbf{p}},\frac{1}{2})$.  Then define $S := \{ \, (\x,1) \, \mid \, \x \in \Z^{r-1} \, \}$ and define $g$ and $h$ similarly in such a way that  $h$ is again skew-symmetric about $(\overline{\mathbf{p}}, \frac{1}{2}) = \mathbf{p}$, but this time $g$  is skew-symmetric about  $(\overline{\mathbf{p}},1)=\mathbf{p}+\frac{1}{2}\mathbf{e}_r$ except possibly on $S$.   Explicitly, for $\x \in \Z^{r-1}$ 
first set $g(2\overline{\mathbf{p}}-\x,-n):=0$.  Then, for $i=n,\ldots,1$, define
$$\begin{array}{lllll}
-h(\x,i+1)	& := 	& h(2\overline{\mathbf{p}}-\x,-i) 	& := & f(2\overline{\mathbf{p}}-\x,-i) - g(2\overline{\mathbf{p}}-\x,-i) \\
-g(\x,i+1) 	& := 	& g(2\overline{\mathbf{p}}-\x,-i+1) & := &   - f(\x,i+1) + h(\x,i+1) \\
\end{array}$$ 
\text{and}
$$\begin{array}{lllll}
-h(\x,1) 	& := 	& h(2\overline{\mathbf{p}}-\x,0) 	& := & f(2\overline{\mathbf{p}}-\x,0) \textcolor{cyan}{-} g(2\overline{\mathbf{p}}-\x,0) \\
-g(\x,1) 			&   	&  	& := & -f(\x,1) + h(\x,1). \\
\end{array}$$ 

Suppose that  $r=1$.   Then  $\mathbf{p}$ is $0$ or $\frac{1}{2}$.  We claim that   in the former case $g$ is skew-symmetric  about $0$ and in the latter $g$ is skew-symmetric about $1$. This follows from the construction of $g$, which immediately gives $g(i)=-g(-i)$ when $\mathbf{p}=0$  and gives $g(i)=-g(-i+1)$ when $\mathbf{p}=\frac{1}{2}$, and the  calculations:
\begin{align*}
g(0)  =  f(0)-h(0) = f(0) +f(1)-g(1) = \cdots   =    \sum_{x \in \Z} f(x)=0  & \textrm{ \ when \ } \mathbf{p}=0, \\
g(1)  =   f(1)-h(1)  =  f(1) +f(0)-g(0) = \cdots = \sum_{x\in\Z} f(x)=0 & \textrm{ \ when \ }  \mathbf{p}=\frac{1}{2}.
\end{align*}
This gives the base case of an induction on $r$.  

Suppose $r>1$.  Redefine  $g$ on $S$ to be everywhere zero. When $\mathbf{p}=(\overline{\mathbf{p}},0)$, this will make $g$ skew-symmetric about $\mathbf{p}$ while $h$ is skew-symmetric about $\mathbf{p}+\frac{1}{2}\mathbf{e}_r$; and when $\mathbf{p}=(\overline{\mathbf{p}},\frac{1}{2})$, it makes $g$ skew-symmetric about $\mathbf{p}+\frac{1}{2}\mathbf{e}_r=(\overline{\mathbf{p}},1)$ while $h$ is skew-symmetric about $\mathbf{p}$. However, $f=g+h$ may now fail on $S$ (and only on $S$).
When $\mathbf{p}=(\overline{\mathbf{p}},0)$,
\begingroup \renewcommand{\arraystretch}{1.4}$$\begin{array}{lcl}
\sum\limits_{\x \in \Z^{r-1}} (f-g-h)(\x,0) & = &  \sum\limits_{\x \in \Z^{r}} (f-g-h)(\x)   \\
 & = &  \sum\limits_{\x \in \Z^{r}} f(\x) - \sum\limits_{\x \in \Z^{r}} g(\x) - \sum\limits_{\x \in \Z^{r}} h(\x) \\
 & = & 0- 0- 0
\end{array}
$$\endgroup 
since  $g$ and $h$ are skew-symmetric and, by hypothesis, $\sum_{\x \in \Z^{r}} f(\x)=0$. Similar calculations show $\sum_{\x \in \Z^{r-1}} (f-g-h)(\x,1) = 0$ when $\mathbf{p}=(\overline{\mathbf{p}},\frac{1}{2})$. Thus, by induction, $f-g-h$ can be expressed as a sum of  $r$   functions that are skew-symmetric about ${\mathbf{p}, \mathbf{p} + \frac{1}{2} \mathbf{e}_1, \ldots, \mathbf{p} + \frac{1}{2} \mathbf{e}_{r-1}}$.    So we have the result.
\end{proof}

In Section \ref{sec:palindromic width} we will need to  express a function $f:\Z^r \to R$ as a sum of functions which are skew-symmetric about $\mathbf{p}, \mathbf{p} +  \mathbf{e}_1, \ldots, \mathbf{p} +  \mathbf{e}_r$. To do so, we  invoke  hypotheses that are stronger than those for Lemma~\ref{lem:skew symmetric functions half}.

Consider the set of vectors $\mathcal{D} = \{\varepsilon_1 \evec_1 + \ldots + \varepsilon_r \evec_r \mid \varepsilon_i = 0,1   \}$ and the $2^r$ double-size grids $2\Z^r+\mathbf{v}$, one for each $\mathbf{v}\in \mathcal{D}$, which partition $\Z^r$.

\begin{lemma}\label{lem:skew symmetric functions}
For all $r \geq 1$ and all $\mathbf{p} \in \Z ^r $, every finitely supported function  $f:\Z^r \to R$ such that $\sum_{\mathbf{x} \in 2\Z^r+\vvec} f(\mathbf{x})=0$ for all $\vvec \in \mathcal{D}$ 
is the sum of $r+1$  finitely supported functions skew-symmetric  about $\mathbf{p}, \mathbf{p} +  \mathbf{e}_1, \ldots, \mathbf{p} +  \mathbf{e}_r$.
\end{lemma}

\begin{proof}
We can express $f$ as the sum 
 $$f (\x) \ = \  \sum\limits_{\vvec \in \mathcal{D}} f_\vvec(\x)$$
where $f_\vvec(\x)=f(\x)$ if $\x \in 2\Z^r+\vvec$ and $0$ otherwise. Let $\varphi_\vvec:\Z^r \to \Z^r$ be given by $\varphi_\vvec(\x)=2\x+\vvec$. For each $\vvec \in \mathcal{D}$ we can apply Lemma \ref{lem:skew symmetric functions half} to $f \circ \varphi_\vvec : \Z^r \to \Z^r$, writing each $f \circ \varphi_\vvec$ as the sum of $r+1$ functions:
		$$f \circ \varphi_\vvec \ = \  \overline{f}_{\vvec,0}+\ldots + \overline{f}_{\vvec,r}$$
where $\overline{f}_{\vvec,i}$ is skew-symmetric about $\frac{1}{2}\mathbf{p}-\frac{1}{2}\vvec+\frac{1}{2}\evec_i$ (taking $\evec_0$ to be the zero-vector). For each $\vvec \in \mathcal{D}$ and each $i \in \{0,\ldots,r\}$, we can write $\overline{f}_{\vvec,i} = f_{\vvec,i} \circ \varphi_\vvec$, where $f_{\vvec,i}$ has support contained in $2\Z^r + \vvec$ and is skew-symmetric about $\varphi_\vvec(\frac{1}{2}\mathbf{p}-\frac{1}{2}\vvec+\frac{1}{2}\evec_i)=\mathbf{p}+\evec_i$.  Note that $f_\vvec=f_{\vvec,0} + \ldots + f_{\vvec,r}$. Thus, since the sum of a family of functions which are all skew-symmetric about the same point will itself be skew-symmetric about that point,  $f$ is the sum of $r+1$ skew-symmetric functions about $\mathbf{p}, \mathbf{p} +  \mathbf{e}_1, \ldots, \mathbf{p} +  \mathbf{e}_r$.
\end{proof}

\subsection{The palindromic width of free metabelian groups}\label{sec:palindromic width}

 In the following, $F=F(x_1,\ldots , x_r)$ is the free group of rank $r$ and $[x_i,x_j]$ denotes $x_ix_jx_i^{-1}x_j^{-1}$. We view $x_1,\ldots , x_r$  as also  generating    $F/F'$ and   $F/F''$, and we identify   $x_i \in F/F' \cong \Z^r$  with the basis vector $\evec_i$.  Given a word $w$ on $x_1^{\pm 1}, \ldots,  x_r^{\pm 1}$, we will denote  the same word read backwards by $\overline{w}$. So $w$ is a palindrome if and only if $w$ is the same word as $\overline{w}$.

While, for the sake of conciseness, we do not use it here,  for visualizing the arguments in this section, we recommend the interpretation of elements of $F/F''$ as \emph{flows} on the Cayley graph of $\Z^r$---see \cite{DLS93,Salemagnus,SZ13,Vershikfreesolvable} for further details. 

The following lemma shows that considering $F'/F''$ suffices for obtaining an upper bound on the   palindromic width of $F/F''$.

\begin{lemma}\label{lem:reduce}
If every  $g \in F'/F''$ is the product of $\ell$ or fewer palindromic words on   $x_1^{\pm1},\ldots,x_r^{\pm1}$, then $$\PalWidth(F/F'',  \{x_1,\ldots,x_r\})  \ \leq \ \ell+r.$$  
\end{lemma}

\begin{proof}  A set of coset representatives for $F'/F''$ in $F/F''$ can be identified with $(F/F'')/(F'/F'') \cong  F/F' \cong\Z^r$, which has palindromic width $r$ with respect to  $x_1,\ldots,x_r$.
\end{proof}

The group $F'/F''$ is the normal closure of  $Y := \{ \, [x_i,x_j]  \,  \mid \, 1 \leq i < j \leq r \, \}$ in $F/F''$. Enumerate  $Y$ as $Y=\{\rho_1,\ldots , \rho_m\}$, where $m = r(r+1)/2$. For all $h \in F'/F''$, there exist finitely supported functions $f_1, \ldots, f_m:F/F' \to \Z$ such that
		\begin{equation} \label{eq:product of relators} h \ = \  \prod\limits_{u \in F/F'} u \rho_1^{f_1{(u)}}\ldots  \rho_m^{f_m{(u)}} u^{-1}.\end{equation}

\begin{rem} \label{rem}
\textup{The reason  the product here is expressed as being over $F/F'$ is that if words $u$ and $v$  on   $x_1^{\pm1},\ldots,x_r^{\pm1}$  are equal in $F/F'$, then $u \rho u^{-1} = v \rho v^{-1}$ in $F/F''$ for all $\rho \in Y$ as elements of $F'$ commute modulo  $F''$.  For the same reason, the order in which the product is evaluated does not effect  the  element of $F'/F''$ it represents.}
\end{rem}

\begin{lemma}\label{lem:skew symmetric gives palindrome}
For $k=1, \ldots, m$ and $\rho_k=[x_i,x_j]$, if $f_k$ is skew-symmetric about ${-\frac{1}{2}(\evec_i+\evec_j)}$, then $h$ can be represented by a  palindrome on $x_1^{\pm 1}, \ldots,  x_r^{\pm 1}$.
\end{lemma}

\begin{proof}
For $k=1, \ldots, m$, let
\begin{equation} \label{hk}
h_k \ := \ \prod\limits_{u \in F/F'} u \rho_k^{f_k{(u)}} u^{-1}.
\end{equation}
Suppose that $\rho_k=[x_i,x_j]$ and  that $f_k$ is skew-symmetric about $-\frac{1}{2}(\evec_i+\evec_j)$. The support of $f_k$  consists of pairs of elements $\uvec$ and $-\uvec-\evec_i-\evec_j$. Enumerate these   so that
		$$\Supp{f_k} \ =\ \{\uvec_1, -\uvec_1-\evec_i-\evec_j, \ldots , \uvec_n\ , -\uvec_n-\evec_i-\evec_j\}.$$
Suppose $u$ and $v$ are words on $x_1^{\pm 1}, \ldots,  x_r^{\pm 1}$ representing $\mathbf{u}$ and $-\mathbf{u} - \mathbf{e}_i -\mathbf{e}_j$, respectively.  Then $v^{-1} = x_i x_j \bar{u}$  and $v = \overline{u^{-1}}x_j^{-1}x_i^{-1}$ in $F/F'$ and  $f_k(v) = -f_k(u)$.  So (see Remark~\ref{rem})
$$v \rho_k^{f_k(v)} v^{-1}   \ = \ 	\overline{u^{-1}}x_j^{-1}x_i^{-1} \rho_k^{-f_k(u)} x_ix_j\overline{u} $$ in $F'/F''$.	Thus, if $u_i$ is a word on $x_1^{\pm 1}, \ldots,  x_r^{\pm 1}$ representing $\uvec_i$, then 
		$$\begin{array}{ll}	
		\Big(u_1 \rho_k^{f_k(u_1)} u_1^{-1}\Big) \ldots \Big(u_n \rho_k^{f_k(u_n)} u_n^{-1}\Big) \Big(\overline{u_n^{-1}}x_j^{-1}x_i^{-1} \rho_k^{-f_k(u_n)} x_ix_j\overline{u_n}\Big)\ldots
		\\
		\qquad  \ldots  \Big(\overline{u_1^{-1}}x_j^{-1}x_i^{-1} \rho_k^{-f_k(u_1)} x_ix_j\overline{u_1}\Big)
		\end{array}$$
represents $h_k$ in $F'/F''$.  
But then, as  $$x_j^{-1}x_i^{-1} \rho_k x_ix_j \ =  \ x_j^{-1}x_i^{-1} \   x_i x_j x_i^{-1} x_j^{-1}   \  x_ix_j  \ = \  x_i^{-1} x_j^{-1}  x_ix_j   \ = \  \overline{\rho_k^{-1}}$$ in $F$ (and so in $F'/F''$),  the palindrome
		$$\Big(u_1 \rho_k^{f_k(u_1)} u_1^{-1}\Big) \ldots \Big(u_n \rho_k^{f_k(u_n)} u_n^{-1}\Big) \Big(\overline{u_n^{-1}} \overline{\rho_k^{f_k(u_n)}}\overline{u_n}\Big) \ldots  \Big(\overline{u_1^{-1}}\overline{\rho_k^{f_k(u_1)}}\overline{u_1}\Big)$$ represents $h_k$  in $F'/F''$.  Let $g_k=\left(u_1 \rho_k^{f_k(u_1)} u_1^{-1}\right) \ldots \left(u_n \rho_k^{f_k(u_n)} u_n^{-1}\right)$. Then, since conjugates of commutators commute in $F' / F''$, the  palindrome
		$$g_m\ldots g_1 \overline{g_1} \ldots \overline{g_m}$$
represents $h$ in $F'/F''$. 
\end{proof}

\begin{cor} \label{cor:conjugate}
If there exists $\mathbf{p} \in \Z^r$ such that each $f_k$ is skew-symmetric about $\mathbf{p}-\frac{1}{2}(\evec_i+\evec_j)$ for each $k$, then $\mathbf{p}^{-1}h\mathbf{p}$ can be represented by a palindrome. 
\end{cor}

Next,  as $F/F' \cong \Z^r$, given suitable conditions on each of the functions $f_k$, we will be able to use   Lemmas~\ref{lem:skew symmetric functions} and \ref{lem:skew symmetric gives palindrome} to reorder the product \eqref{eq:product of relators} representing $h$ in $F'/F''$,  to express $h$ as a product of boundedly many palindromes. Recall that $\mathcal{D} = \{\varepsilon_1 \evec_1 + \ldots + \varepsilon_r \evec_r \mid \varepsilon_i = 0,1  \}$.

\begin{lemma}\label{lem:palindromic length when delta sums zero}
Suppose $h \in F' / F''$ is such that $\sum\limits_{\x \in 2\Z^r+\vvec}f_k(\x)=0$ for all $\vvec \in \mathcal{D}$ and all $k$. Then in $F' / F''$, $h$ is a product of $3r+1$ or fewer  palindromes on $x_1^{\pm 1}, \ldots,  x_r^{\pm 1}$.
\end{lemma}

\begin{proof}
Let  $h_k$ be as in   \eqref{hk}. By Lemma \ref{lem:skew symmetric functions}, $f_k$ is the sum of $r+1$ skew-symmetric functions, $f_k=f_k^{(0)}+\ldots + f_k^{(r)}$, where $f_k^{(0)}$ is skew-symmetric about $-\frac{1}{2}(\evec_i+\evec_j)$ and $f_k^{(\alpha)}$ is skew-symmetric about $-\frac{1}{2}(\evec_i+\evec_j) + \evec_\alpha$ for   $\alpha=1,\ldots ,r$. Let  
		$$h^{(\alpha)} :=  \prod\limits_{u \in F/F'} u \rho_1^{f_1^{(\alpha)}{(u)}} \ldots \rho_m^{f_m^{(\alpha)}{(u)}} u^{-1}.$$
Then $h = h^{(0)} \cdots h^{(r)}$ in $F'/F''$.  

Let  $x_0$ denote the identity and $\mathbf{e}_0$ the zero-vector.  
For $\alpha=0, \ldots , r$,  by Corollary~\ref{cor:conjugate},
$h^{(\alpha)}=x_\alpha p^{(\alpha)} x_\alpha^{-1}$ in $F'/F''$ for some palindrome $p^{(\alpha)}$,
since $f_k^{(\alpha)}$ is skew-symmetric about $\evec_\alpha-\frac{1}{2}(\evec_i+\evec_j)$.   Then $$ p^{(0)} \left(x_1p^{(1)}x_1^{-1} \right)  \ldots \left( x_rp^{(r)}x_r^{-1}\right)$$ represents $h$  in $F'/F''$ and is the product of $3r+1$ palindromes.
\end{proof}

Next we prove a version of Lemma~\ref{lem:palindromic length when delta sums zero} free of the hypotheses on $f_k$.

\begin{lemma}\label{lem:palindromic length in F'/F''}
Every  $h \in F'/F''$ is a product of  $$2^{r-1}r(r+1)(2r+3)+3r+1$$ or fewer palindromes on $x_1^{\pm 1}, \ldots,  x_r^{\pm 1}$.  
\end{lemma}

\begin{proof}
For $k=1,\ldots , m$ and $\mathbf{v}=\varepsilon_1 \evec_1 + \ldots + \varepsilon_r \evec_r   \in \mathcal{D}$, let $D_{k,\vvec}:=\sum\limits_{\x \in 2\Z^r+\vvec}f_k(\x)$ and, if $\rho_k = [x_i,x_j]$, define the \emph{battlement} words  
$$q_{k,\vvec} \ := \  x_1^{\varepsilon_1} \ldots x_r^{\varepsilon_r} \ (x_jx_ix_j^{-1}x_i)^{D_{k,\vvec}}x_i^{-2D_{k,\vvec}} \ x_r^{-\varepsilon_r}\ldots x_1^{-\varepsilon_1}.$$
Each $q_{k,\vvec}$ is the product of at most $2r+3$ palindromes. Indeed, for  $D \in \N$, $(x_jx_ix_j^{-1}x_i)^D$ is the product of two palindromes: $x_j$ and $x_ix_j^{-1}x_ix_j\ldots x_jx_ix_j^{-1}x_i$. Enumerate $\mathcal{D}$ as $\{\vvec_1,\ldots , \vvec_{2^r} \}$. Define 
		$$q \ := \ q_{1,\vvec_1} \ldots q_{1,\vvec_{2^r}} \ldots q_{m,\vvec_1} \ldots q_{m,\vvec_{2^r}}.$$
In $F$,  $$q_{k,\vvec} \ = \    x_1^{\varepsilon_1} \ldots x_r^{\varepsilon_r} \  \rho_k^{-1}  \left(  x_i^2  \rho_k^{-1} x_i^{-2} \right) \cdots  \left( x_i^{2D_{k,\vvec}-2}  \rho_k^{-1} x_i^{-(2D_{k,\vvec}-2)} \right)    \ x_r^{-\varepsilon_r}\ldots x_1^{-\varepsilon_1},$$ a product of $D_{k,\vvec}$ conjugates of $\rho_k^{-1}$ by elements of $2 \Z^r + \vvec$.  
So, multiplying by $q$ corrects each function $f_k$ suitably  so that $hq$ can be represented by a word as in \eqref{eq:product of relators} to which Lemma~\ref{lem:palindromic length when delta sums zero} applies. 

Each of the $r(r+1)2^{r-1}$  battlement words comprising $q$ costs at most $2r+3$ palindromes.  So $q$, and hence also $q^{-1}$, is a product of ${2^{r-1}r(r+1)(2r+3)}$ or fewer palindromes. Since $hq$ can be expressed as the product of   $3r+1$ or fewer palindromes in $F'/F''$, the result follows.
\end{proof}

\begin{proof}[Proof of Theorem~\ref{thm:main}]
Combine Lemmas~\ref{lem:reduce} and \ref{lem:palindromic length in F'/F''}.
\end{proof}

\section{Finite palindromic width of solvable groups satisfying max-n}\label{sec:max-n}

The normal closure  $\langle \! \langle X \rangle \! \rangle$  of a subset $X$ of a group $G$ is the smallest normal subgroup of $G$ containing $X$.  
So $G$ satisfies the \emph{maximal condition on normal subgroups} (or \emph{max-n}) if for every   $N \trianglelefteq G$, there exists a finite  $X \subseteq G$ such that $N = \langle \! \langle X \rangle \! \rangle$. 


\begin{proof}[Proof of Theorem~\ref{thm:max n}]
Let $G$ be a solvable group of derived length $d$ satisfying max-n with finite generating set $A$. Suppose that the $d$--th derived subgroup of $G$ is $\langle \! \langle A_d \rangle \! \rangle$ for some finite $A_d \subseteq G$.  Extend $A$ to the possibly larger, but still finite, generating set $B = A \cup A_1 \cup  \cdots \cup A_{d-1}$ of $G$. The following result  gives   an expression for an element of the derived subgroup $G'$ of $G$.

\begin{lemma}[Akhavan-Malayeri \cite{Akha06}]
There exists $K>0$, depending on the size of $B$, such that any element of $G'$ can be expressed as the product of $K$ or fewer commutators of the form $[g,b]$, or their conjugates, where $b \in B$.
\end{lemma}

The following two observations can both be found in \cite[Lemma 2.5]{BG13}.
 Each commutator $[g,b]$ is the product of three palindromes, namely
		$$gbg^{-1}b^{-1} = ( g b \overline{g} )(\overline{g}^{-1} g ^{-1}) (b^{-1}).$$
Conjugation  increases palindromic length by at most 1. Indeed, if $g=g_1\ldots g_{2k}$, where each $g_i$ is a palindrome and $g_{2k}$ is possibly the empty word, then, for $h \in G$,
		$$hgh^{-1} = (hg_1\overline{h})(\overline{h}^{-1} g_2 h^{-1} ) (h g_3 \overline{h} ) \ldots ( \overline{h}^{-1} g_{2k} h^{-1} ).$$

So, every element of $G'$ may be written as the product of at most $4K$ palindromes. Finally, $G/G'$ is a finitely generated abelian group, so has  palindromic width equal to the size of a minimal generating set. So $\PalWidth(G,B)$ is finite.  
\end{proof}

\section{Open questions} \label{Open questions}

Quantitative results concerning the palindromic width of free nilpotent groups with respect to particular generating sets have recently been established \cite{BG13}.  However the relationship between palindromic width and the choice of finite generating set remains unclear.  In particular:

\begin{question}
Is there  a group $G$ with finite generating sets $X$ and $Y$ such that $\PalWidth(G, X)$ is finite, but $\PalWidth(G, Y)$ is infinite?  
\end{question}

A difficultly here may be a shortage of known obstructions to palindromic width being finite.  The quasi-morphism approach in  \cite{BST05, BT06} does not appear to transfer readily to other groups.

\begin{question}
Do finitely generated solvable groups of higher derived length have finite palindromic width with respect to some (or all) finite generating sets? 
\end{question}

The methods used in this paper for proving Theorem~\ref{thm:main} have the potential to be applied to a larger class of finitely generated groups of higher derived length. In particular, one may consider generalising Lemma~\ref{lem:skew symmetric functions half} to functions $f:G \to R$ where $G$ is not abelian. For example, taking $G$ to be polycyclic seems to be a suitable area to experiment. However, if the factor groups of the derived series of $G$ include infinite-rank abelian groups then it is not clear whether this will be possible.

 Consider  a group $G$ with finite commutator width.  If, with respect to some finite generating set $A$, every commutator has finite palindromic length, then $\PW(G, A) < \infty$. After all $G/G'$ is a finitely generated abelian group, and so modulo $G'$ every element of $G$ has palindromic with at most $|A|$.  This  approach (but specialized to particular commutators) is the basis of  Bardakov \& Gongopadhyay's proof of Theorem~\ref{thm:main} and our proof of Theorem~\ref{thm:max n}.   It motivates:

\begin{question}[Bardakov \& Gongopadhyay \cite{BG13}]
Does a group $G$ have finite palindromic width with respect to some finite generating set  precisely when it has finite commutator width? 
\end{question}

Precise values of $\PW(G, A)$ appear generally elusive.  In the context of this paper an instance one might pursue is:

\begin{question}
What is the palindromic width of the free metabelian group $F/F''$ of rank $r$ with respect to its standard set of $r$ generators?
\end{question}

Finally we ask:

\begin{question}
For which  normal subgroups $N$ of a finite-rank free group $F$ does  $F/N$ having finite palindromic width imply the same of $F/N'$? 
\end{question}

In Section~\ref{metabelian groups section} we answered this affirmatively when $N=F'$. The elements of $F/N'$ can be described as flows on a Cayley graph of $F/N$ \cite{DLS93}. If these flows are suitably symmetric, then they determine a palindromic element of $F/N'$.

\bibliography{bibliography_pal}{}

\begin{thebibliography}{10}

\bibitem{Akha06}
M.~Akhavan-Malayeri.
\newblock Commutator length of solvable groups satisfying max-{$n$}.
\newblock {\em Bull. Korean Math. Soc.}, 43(4):805--812, 2006.

\bibitem{Akha2010wreath}
M.~Akhavan-Malayeri.
\newblock On commutator length and square length of the wreath product of a
  group by a finitely generated abelian group.
\newblock In {\em Algebra Colloq.}, volume~17, pages 799--802. World
  Scientific, 2010.

\bibitem{AMR98}
M.~Akhavan-Malayeri and A.~Rhemtulla.
\newblock Commutator length of abelian-by-nilpotent groups.
\newblock {\em Glasgow Math. J.}, 40(1):117--121, 1998.

\bibitem{BGsolvablepals}
V.~Bardakov and K.~Gongopadhyay.
\newblock Palindromic width of finitely generated solvable groups.
\newblock {\em Comm. Algebra}.
\newblock To appear.

\bibitem{BGnilpotentpals2}
V.~Bardakov and K.~Gongopadhyay.
\newblock On palindromic width of certain extensions and quotients of free
  nilpotent groups.
\newblock {\em Internat. J. Algebra Comput.}, 24(05):553--567, 2014.

\bibitem{BG13}
V.~Bardakov and K.~Gongopadhyay.
\newblock Palindromic width of free nilpotent groups.
\newblock {\em J. Algebra}, 402:379--391, 2014.

\bibitem{BST05}
V.~Bardakov, V.~Shpilrain, and V.~Tolstykh.
\newblock On the palindromic and primitive widths of a free group.
\newblock {\em J. Algebra}, 285(2):574--585, 2005.

\bibitem{BT06}
V.~Bardakov and V.~Tolstykh.
\newblock The palindromic width of a free product of groups.
\newblock {\em J. Aust. Math. Soc.}, 81(2):199--208, 2006.

\bibitem{BHV}
B.~Bekka, P.~de~la Harpe, and A.~Valette.
\newblock {\em Kazhdan's property ({T})}, volume~11 of {\em New Mathematical
  Monographs}.
\newblock Cambridge University Press, Cambridge, 2008.

\bibitem{BS78}
R.~Bieri and R.~Strebel.
\newblock Almost finitely presented soluble groups.
\newblock {\em Comment. Math. Helv.}, 53(2):258--278, 1978.

\bibitem{CK83}
D.~Carter and G.~Keller.
\newblock Bounded elementary generation of {${\rm SL}_{n}({\mathcal O})$}.
\newblock {\em Amer. J. Math.}, 105(3):673--687, 1983.

\bibitem{DLS93}
C.~Droms, J.~Lewin, and H.~Servatius.
\newblock The length of elements in free solvable groups.
\newblock {\em Proc. Amer. Math. Soc.}, 119(1):27--33, 1993.

\bibitem{ENS}
I.~Erovenko, N.~Nikolov, and B.~Sury.
\newblock Bounded generation and second bounded cohomology of wreath products.
\newblock www.isibang.ac.in/$\sim$sury/bgsbcwreath.pdf.

\bibitem{FM01}
B.~Farb and L.~Mosher.
\newblock On the asymptotic geometry of abelian-by-cyclic groups.
\newblock {\em Acta Math.}, 184(2):145--202, 2000.

\bibitem{FinkPalindromes}
E.~Fink.
\newblock Palindromic width of some wreath products.
\newblock arXiv:1402.4345.

\bibitem{Hall54}
P.~Hall.
\newblock Finiteness conditions for soluble groups.
\newblock {\em Proc. London Math. Soc. (3)}, 4:419--436, 1954.

\bibitem{Kr84}
P.~H. Kropholler.
\newblock On finitely generated soluble groups with no large wreath product
  sections.
\newblock {\em Proc. London Math. Soc. (3)}, 49(1):155--169, 1984.

\bibitem{Mu95}
V.~K. Murty.
\newblock Bounded and finite generation of arithmetic groups.
\newblock In {\em Number theory ({H}alifax, {NS}, 1994)}, volume~15 of {\em CMS
  Conf. Proc.}, pages 249--261. Amer. Math. Soc., Providence, RI, 1995.

\bibitem{Salemagnus}
A.~W. Sale.
\newblock Metric behaviour of the magnus embedding.
\newblock {\em Geom. Dedicata}, pages 1--9, 2014.

\bibitem{SZ13}
L.~Saloff-Coste and T.~Zheng.
\newblock Random walks on free solvable groups.
\newblock arxiv:1307.5332, 2013.

\bibitem{Segal09}
D.~Segal.
\newblock {\em Words: notes on verbal width in groups}, volume 361 of {\em
  London Mathematical Society Lecture Note Series}.
\newblock Cambridge University Press, Cambridge, 2009.

\bibitem{Stro66}
P.~Stroud.
\newblock {\em Topics in the theory of verbal subgroups}.
\newblock PhD thesis, University of Cambridge, 1966.

\bibitem{Su97}
B.~Sury.
\newblock Bounded generation does not imply finite presentation.
\newblock {\em Comm. Algebra}, 25(5):1673--1683, 1997.

\bibitem{Ta90}
O.~I. Tavgen{\cprime}.
\newblock Bounded generability of {C}hevalley groups over rings of
  {$S$}-integer algebraic numbers.
\newblock {\em Izv. Akad. Nauk SSSR Ser. Mat.}, 54(1):97--122, 221--222, 1990.

\bibitem{Vershikfreesolvable}
A.~M. Vershik.
\newblock Geometry and dynamics on the free solvable groups.
\newblock Preprint. Erwin Schroedinger Institute, Vienna, 1999, pp. 1–16.

\end{thebibliography}
\bibliographystyle{plain}

\small{ 
\noindent  \textsc{Timothy R.\ Riley} \rule{0mm}{6mm} \\
Department of Mathematics,
Cornell University, 310 Malott Hall, Ithaca, NY 14850, USA \\ \texttt{tim.riley@math.cornell.edu}, \
{http://www.math.cornell.edu/$\sim$riley/}}

\small{ 
\noindent  \textsc{Andrew W.\ Sale} \rule{0mm}{6mm} \\
Department of Mathematics,
Vanderbilt University, 1326 Stephenson Center, Nashville, TN 37240, USA\\ \texttt{andrew.sale@some.oxon.org}, \
{http://perso.univ-rennes1.fr/andrew.sale/}}

\end{document}